\documentclass[a4wide,12pt]{article}

\usepackage{lipsum}
\usepackage{amsfonts}
\usepackage{graphicx}
\usepackage{epstopdf}
\usepackage{cite}
\ifpdf
  \DeclareGraphicsExtensions{.eps,.pdf,.png,.jpg}
\else
  \DeclareGraphicsExtensions{.eps}
\fi

\usepackage{subcaption}
\usepackage{calc}
\usepackage{microtype}
\usepackage{amsmath}
\usepackage{amssymb}
\usepackage{amsthm}
\usepackage{color}

\usepackage{graphicx}
\usepackage{caption}
\usepackage{subcaption}
\usepackage[hidelinks]{hyperref}
\hypersetup{
    colorlinks=true,
    linkcolor=blue,
    filecolor=blue,      
    urlcolor=blue,
    citecolor = blue,
    }
\usepackage{tikz}
\usetikzlibrary{plotmarks}
\usepackage{pgfplots}
\usepackage{dsfont}
\usepackage{bm}
\usepackage{listings}
\usepackage{cleveref}
\usepackage{comment}
\usepackage{cite}
\usepackage{mathrsfs,graphicx,color,float, indentfirst,textcomp}
\usepackage{setspace}
\usepackage{latexsym,lscape,amsfonts}
\usepackage{enumerate}
\usepackage[shortlabels]{enumitem}
\usepackage{lscape}
\usepackage{multirow} 
\usepackage{xcolor}
\usepackage{geometry}
\usepackage{mathtools}
\usepackage[normalem]{ulem}
 \usepackage{algorithmic}
\usepackage[algoruled,boxed,lined]{algorithm2e}
\numberwithin{equation}{section}
\newtheorem{theorem}{Theorem}[section]
\newtheorem{remark}[theorem]{Remark}
\newtheorem{definition}[theorem]{Definition}
\newtheorem{lemma}{Lemma}[section]
\newtheorem{corollary}[theorem]{Corollary}

\newtheorem{assumption}[theorem]{Assumption}

\title{Wasserstein convergence rates for stochastic particle approximation
of Boltzmann models}

\author{Giacomo Borghi\thanks{Maxwell Institute for Mathematical Sciences and Department of Mathematics, School of Mathematical and Computer Sciences (MACS), Heriot-Watt University, Edinburgh, UK
  (\texttt{g.borghi@hw.ac.uk}, \texttt{l.pareschi@hw.ac.uk}).}
\and Lorenzo Pareschi\footnotemark[2]$\;\,$\thanks{Department of Mathematics and Computer Science, University of Ferrara, Italy.}
}

\usepackage{amsopn}

\usepackage{mathrsfs}
\usepackage{mathtools}
\usepackage{amssymb}
\usepackage[shortlabels]{enumitem}
\usepackage{bm}

\newcommand{\Rd}{ {\mathbb{R}^d}}
\newcommand{\Rm}{ {\mathbb{R}^m}}
\newcommand{\R}{ \mathbb{R}}
\newcommand{\E}{\mathcal{E}}
\newcommand{\D}{\Omega}

\newcommand{\ve}{\varepsilon}
\renewcommand{\d}{\textup{d}}

\newcommand{\supp}{\textup{supp}}
\newcommand{\law}{\textup{Law}}

\newcommand{\unif}{\textup{Unif}}
\newcommand{\bern}{\textup{Bern}}

\newcommand{\coll}{\mathscr{C}}
\renewcommand{\b}{B}

\renewcommand{\j}{\textup{\textbf{j}}}

\renewcommand{\v}{\textup{\textbf{v}}}

\newcommand{\W}{\mathbb{W}}
\newcommand{\BV}{\textup{\textbf{{V}}}}
\newcommand{\BW}{\textup{\textbf{{W}}}}

\newcommand{\fdt}{f^{\Delta t}}
\newcommand{\KR}{\mathtt{KR}}
\newcommand{\BL}{\mathtt{BL}}

\ifpdf
\hypersetup{
  pdftitle={},
  pdfauthor={G. Borghi, L. Pareschi}
}
\fi

\begin{document}

\maketitle

\begin{abstract}
We establish quantitative convergence rates for stochastic particle approximation based on Nanbu-type Monte Carlo schemes applied to a broad class of collisional kinetic models. Using coupling techniques and stability estimates in the Wasserstein-1 (Kantorovich–Rubinstein) metric, we derive sharp error bounds that reflect the nonlinear interaction structure of the models. Our framework includes classical Nanbu Monte Carlo method and more recent developments as Time Relaxed Monte Carlo methods. The results bridge the gap between probabilistic particle approximations and deterministic numerical error analysis, and provide a unified perspective for the convergence theory of Monte Carlo methods for Boltzmann-type equations. As a by-product, we also obtain existence and uniqueness of solutions to a large class of Boltzmann-type equations.
\end{abstract}

\bigskip

\textbf{keywords}
Monte Carlo methods, Boltzmann equation, Wasserstein
distance, Nanbu method, Time Relaxed Monte Carlo, convergence estimates

\bigskip
\textbf{MSCcodes}
65C05, 35Q20, 82C22, 49Q22, 65M15

\section{Introduction}
The Boltzmann equation is a foundational model in non-equilibrium statistical mechanics, describing the evolution of dilute gases by linking microscopic interactions to macroscopic fluid dynamics^^>\cite{cercignani, CIP}. Since its introduction, it has prompted fundamental mathematical investigations into well-posedness, asymptotic behavior, and hydrodynamic limits. Its high dimensionality and nonlinearity, however, pose significant challenges for both analysis and computation.

Among numerical approaches, Direct Simulation Monte Carlo methods^^>\cite{bird1976molecular} and Nanbu-type algorithms^^>\cite{nanbu1980} have become essential for approximating solutions to the Boltzmann equation^^>\cite{caflisch1998, pareschi2001intro,wagner2005book}. These particle-based methods simulate stochastic binary collisions and are valued for their simplicity, scalability, and effectiveness in high-dimensional regimes. Originally developed for applications in gas dynamics and plasma physics, they have since found broader use in areas such as social dynamics, biological systems, and data science, where kinetic models describe collective behavior in abstract spaces^^>\cite{partos13, benfenati2021binary, tosin2019traffic,cordier2005market,toscani2006opinion,boudin2010opinion,bisi2009,AlPa}.


Despite their practical success, a rigorous convergence analysis in the context of numerical approximation has remained incomplete.
While consistency results exist under restrictive assumptions^^>\cite{babovsky1989full,pulvirenti1994,wagner2005book,myong2019review,wagner1992bird}, quantitative error estimates in metrics such as Wasserstein distances are still lacking. 

The purpose of this work is to address this gap by establishing a general framework for quantifying the convergence of Nanbu-type particle systems for a broad class of collisional kinetic equations, including the Boltzmann equation for specific interaction kernels. To our knowledge, this is the first work to derive Wasserstein convergence rates for Nanbu-type schemes beyond consistency and for general Lipschitz collision maps, thereby contributing a rigorous numerical analysis framework for a broad class of Monte Carlo solvers for Boltzmann-type equations. The analysis relies on coupling techniques, moments control, and probabilistic tools from optimal transport theory. We derive explicit error bounds in the Wasserstein-1 distance and extend our results to advanced algorithms such as Time Relaxed Monte Carlo methods^^>\cite{pareschi2001trmc}. 

\subsection{Collisional kinetic equations and Nanbu method}
We are interested in evolutionary equations describing a system of particles, sometimes called \textit{agents}, undergoing interactions of binary type in a domain $\D\subseteq\Rd$, which is possibly unbounded. At every instant of time, any couple of particles $v,v_* \in \D$ may interact and change their state to $v',v'_* \in \D$. The interaction, or \textit{collision}, depends on additional parameters $\theta, \theta_*\in \Theta$, and is determined by a collision map $\coll: \D \times \D\times \Theta \to \Rd$:
\begin{equation}
\begin{split}
v' &= \coll(v,v_*, \theta) \\
v_*'& = \coll(v_*,v,\theta_*)\,.
\end{split} \label{eq:coll}
\end{equation}
The parameters $\theta, \theta_*$ are assumed to be distributed according to a probability density $\b = \b(\theta)$, and might be dependent on each other. Let $f= f(v,t)$ with $\int f(v,t) \d v = 1$ be the particle density at time $t\geq 0$. The evolution of the system is described the integro-differential equation of Boltzmann type
\begin{equation} \label{eq:boltzmann} \begin{dcases}
\frac{\partial f}{\partial t }(v,t) = Q(f,f)(v,t) &\quad v\in \Omega,\, t\geq 0 \\
f(v,0)  = f_0(v) &\quad v \in \Omega
\end{dcases}
\end{equation}
where the collisional operator $Q(\cdot,\cdot)$ is given by 
\begin{small}
\begin{equation} \label{eq:operator}
\int \phi(v) Q(f,f)(v,t)\, \d v = \frac12\! \iiint\! \b(\theta) \left( \phi(v') + \phi(v'_*) - \phi(v) - \phi(v_*)\right)f(v,t) f(v_*,t) \d \theta\,\d v\,\d v_*
\end{equation}\end{small}
for any bounded and continuous test function $\phi$.

This type of equations is widely used to model different types of complex interacting systems^^>\cite{partos13}. Prominent examples are the Kac's model^^>\cite{kac1956} in physics, kinetic models in economics for wealth distribution^^>\cite{cordier2005market}, models in social sciences for opinion formation^^>\cite{toscani2006opinion}, and in particle-based optimization^^>\cite{benfenati2021binary}. The homogeneous Boltzmann equation also falls into this class, for specific interaction kernels. We refer to Section \ref{sec:models} for more details and an overview of the kinetic models taking the form \eqref{eq:boltzmann} with collisional operator \eqref{eq:operator}.

Differently form Direct Simulation Monte Carlo (DSMC) methods^^>\cite{bird1976molecular} which aim to simulate the microscopic dynamics of the particles, the Nanbu algorithm^^>\cite{nanbu1980} is designed to provide a numerical solution to \eqref{eq:boltzmann}. 
Let  $\Delta t \in (0,1]$ be a time step, we consider the forward Euler discretization of  \eqref{eq:boltzmann} given by $f_{n+1}  = f_n + \Delta t Q(f_n,f_n)$. 
Note that the collisional operator \eqref{eq:operator} can be written as $Q(f,f) = Q^+(f,f) - f$, where $Q^+$ corresponds to the positive contribution of the new particles generated via collision, so that the update can be equivalently formulated as
\begin{equation} \label{eq:euler}
f_{n+1} = (1-\Delta t) f_{n} + \Delta t Q^+(f_n,f_n)\,.
\end{equation}
The Nanbu method aims to approximate the Euler scheme $f_n$ with an ensemble of $N$ particles $V_n^i$, $i = 1,\dots,N$, 
\begin{equation*}
f_n \approx f^N_n \qquad \textup{where} \qquad f^N_n :=\frac1N \sum_{i=1}^N \delta_{V_n^i}\;,
\end{equation*}
and their update is based on a probabilistic interpretation of \eqref{eq:euler}.
Each particle at time step $n+1$ is either taken from $f^N_n$ (with probability $\Delta t$), or it is the result of a collision, and sampled from $Q^+(f_n^N,f^N_n)$ (with probability $1 - \Delta t$). Notably, the Nanbu Monte Carlo approximation of the collisional operator is designed to maintain a computational cost of $ \mathcal{O}(N) $, despite its quadratic structure. We refer to Algorithm \ref{alg:nanbu} for a precise description of the particle update strategy.

\begin{algorithm}[t]
\caption{Nanbu Monte Carlo method}
\label{alg:nanbu}
\begin{algorithmic}
\STATE{Sample $N$ particles $V^i_0$, $i = 1, \dots,N$ from $f_0$}
\STATE{$n = 0$}
\WHILE{$n\Delta t < T$}
\STATE{with probability $1 - \Delta t:$}
\STATE{\hspace{1cm}  $V_{n+1}^i = V_{n}^i$ }
\STATE{with probability $\Delta t:$}
\STATE{\hspace{1cm}  select a random particle $j$}
\STATE{\hspace{1cm}  sample a parameter $\theta$ from $\b(\theta)$}
\STATE{\hspace{1cm}  $V^i_{n+1} =\coll(V_n^i, V_n^j, \theta)$}
\STATE{$n = n+1$}
\ENDWHILE
\end{algorithmic}
\end{algorithm}

\subsection{Contribution of the paper} 

The aim of the paper is to analyze the error introduced by the Nanbu method in terms of number  $N$ of particles used in the simulation and the time step $\Delta t$. We also extend the analysis to Time Relaxed Monte Carlo variants^^>\cite{pareschi2001trmc}.

To quantify the convergence, we employ the Kantorovich--Rubinstein norm^^>\cite{kantorovich1958} for signed Radon measures
\[
\| \mu \|_\KR:= \sup_{\|\phi\|_{\textup{Lip}}\leq 1} \int_{\Rd} \phi(v)\mu(\d v)\,\qquad \mu\in\mathcal{M}(\Rd)
\]
which extends the Wasserstein-1 distance  between probability measures with finite first moments, $\W_1(f,g) = \| f- g\|_\KR$^^>\cite{villani2009}. 
Under the assumption of $\coll$ being Lipschitz and growing linearly (Assumption \ref{asm:coll}), and $B$ having bounded support, we claim the following:

\begin{itemize}
\item  (Theorem \ref{t:nanbu}) The Nanbu particle system is as efficient as a system of true i.i.d. particles in approximating the forward Euler scheme:
\[
\sup_{n\Delta t \leq T}\mathbb{E}\|f^N_n - f_n\|_\KR \lesssim \ve(N) 
\]
with $\ve(N)$ being the  approximation error of $N$ i.i.d. $f_n$-distributed particles.
\item (Theorem \ref{t:euler}, Corollary \ref{c:full}) There exists a unique weak measure solution $f$ (see Definition \ref{def:weaksol}) to the Cauchy problem \eqref{eq:boltzmann}
and it holds
\[
\sup_{t_n = n \Delta t \leq T}\mathbb{E} \|f^N_n - f(t_n)\|_\KR \lesssim  \Delta t  + \ve(N) \,.
\]
\item (Theorem \ref{t:trmc}) Since the Monte Carlo strategy used in Algorithm \ref{alg:nanbu} to sample from $Q^+(f^N_n,f^N_n)$ is independent of the numerical time discretization of the kinetic model, we extend the convergence in $N$ to a larger class of schemes beyond forward Euler. We show in particular that first-order Time Relaxed Monte Carlo methods (Algorithm \ref{alg:trmc}) converge with optimal rate $\ve(N)$ to the corresponding discretized dynamics.
\end{itemize}

\subsection{Literature review}

The convergence properties of the Nanbu scheme towards the explicit Euler discretization of the space homogeneous Boltzmann equation have been first studied in^^>\cite{Babovsky1989ACP}. The author proves that if $V^i_0$, $i = 1,\dots,N$ are $f_0$-distributed and i.i.d. then, for the first step $n = 1$, we have weak convergence (in duality with the set of bounded continuous functions) of the empirical measure:
\begin{equation*}
f^N_1 \rightharpoonup f_1 \qquad \textup{as}\quad  N \to \infty \quad  \textup{in probability}.
\end{equation*}
This result cannot be iterated to prove convergence for the subsequent iterations $n>1$ and, therefore, it can be considered a consistency result for the Nanbu particle method, see also^^>\cite[Section 3.5.4]{wagner2005book} for a discussion on this aspect. With the same technique, the result was then extended to the full Boltzmann equation in^^>\cite{babovsky1989full}. Another proof of consistency was proposed in^^>\cite{wagner2005book} using the Bounded Lipschitz norm $\| \cdot\|_{\BL}$ (which metrizes weak convergence like $\|\cdot\|_\KR$). In particular, the authors show that 
$
\lim_{N \to \infty} \mathbb{E}  \| f_1 - f_1^N\|_{\BL}= 0\,.
$
These consistency results have also been extended to the Nanbu--Babovsky Monte Carlo strategy^^>\cite{Babovsky1989ACP}, which is an exactly conservative variant of the Nanbu strategy.

If we broad our viewpoint by considering DSMC methods for the Boltzmann equation, in particular Bird's particle method ^^>\cite{bird1976molecular}, we note that theoretical analysis available in the literature lacks to prove quantitative convergence over an entire time window providing a convergence rate in $N$, see for instance^^>\cite{wagner1992bird,pulvirenti1994,wagner2003} and the more recent review^^>\cite{myong2019review}. Therefore, up to our knowledge, a quantitative convergence result as the ones in Theorem \ref{t:nanbu}, represent a novelty in the analysis of Monte Carlo numerical methods for collisional kinetic equations.

The error analysis we propose follows a different line of research in kinetic theory, initiated by Kac^^>\cite{kac1956}, which aims to derive the Boltzmann equation as the mean-field limit of a particle system of jump processes. The stochastic particles are assumed to evolve in a time-continuous settings and undergo binary collisions. Proving convergence towards the Boltzmann kinetic model as $N \to \infty$ is related to the \textit{propagation of chaos} property of the system^^>\cite{diez2022review1,diez2022review2}, which, loosely speaking, states that particles become uncorrelated as $N\to \infty$. We refer to^^>\cite{mischler2013program} for an overview on the subject, and to^^>\cite{heydecker2022} for updated references. 
Particularly relevant for this paper is the work^^>\cite{fournier2016} where the authors prove propagation of chaos, with a rate, of a time-continuous particle system undergoing Nanbu collisions.  The interaction is of Nanbu type in the following sense: at time $t$, particle $i$ might change state due to a collision with particle $j$, but particle $j$ does not necessarily change. The authors are able to prove chaos propagation via a coupling technique and Wasserstein-2 distance for hard potentials and Maxwellian particles. 

The work^^>\cite{cortez2016} considers time-continuous particle systems with Lipschitz binary interactions, where particles collide \textit{as couples}: unlike Nanbu's collisions, if particle $i$ collides with particle $j$, then particle $j$ collides with particle $i$ too (as in the Nanbu--Babovsky algorithm). The authors of ^^>\cite{cortez2016} are able to show that the chaos propagates also in this case with a similar technique to^^>\cite{fournier2016}. The result has been extended to Maxwellian molecules undergoing classical Boltzmann collisions in^^>\cite{cortez2018}. 

We conclude the literature review by mentioning that first order convergence in time of the Nanbu method was shown in^^>\cite{babovsky1989full}, and also in^^>\cite{wagner2005book}. Clearly, this is what we expected, as the method is based on the simulation via particles of the explicit Euler discretization  \eqref{eq:euler}. Error estimates on Forward and Backward Euler discretizations of the Boltzmann equation were studied in^^>\cite{mischler1999spatially} using $L^1$ norms. Time Relaxed Monte Carlo scheme based on higher order asymptotic-preserving time discretization of the Boltzmann equation show faster and uniform convergence with respect to the time step $\Delta t$^^>\cite{gabetta1997}. Well-posedness of the Boltzmann equation in a measure framework was studied in^^>\cite{lu2012measure}. The error analysis we propose for the forward Euler scheme provides a well-posedness result for the Cauchy problem \eqref{eq:boltzmann}, which, to our knowledge, is lacking in the literature 
(despite being sensibly easier compared with the classical Boltzmann equation).

\subsection{Outline of the paper} We first present in Section \ref{sec:results} the settings and the main results of the paper, including an overview of the models considered. Section \ref{sec:prel} presents the necessary notation for the proofs and  recall some results we will use throughout the paper.  Sections \ref{sec:nanbu} is devoted to the convergence of the Nanbu Monte Carlo methods to the forward Euler scheme, and the extension to  Time Relaxed Monte Carlo methods (Section \ref{sec:trmc}).
In Section \ref{sec:euler} we show well-posedeness of the kinetic equations considered and quantitative convergence of the forward Euler scheme.  We conclude the paper with an outlook on future research perspectives in Section \ref{sec:end}.

\section{Main results}
\label{sec:results}

\subsection{Assumptions and examples}
\label{sec:models}

The kinetic equation \eqref{eq:boltzmann} with collisional operator $Q$  given by \eqref{eq:operator} is a flexible model that can describe many particle or multi-agent systems undergoing binary collisions. Each model is then characterized 
by a different collisional map $\coll$ which determines the microscopic dynamics of the interaction. The error analysis will cover models where $\coll$ satisfy a Lipschitz and growth assumption:

\begin{assumption}
\label{asm:coll}
There exist constants $L_\coll, C_\coll>0$ such that for any $v,v_*,w,w_*$$\in \D$ and $\theta\in \Theta$ 
\begin{align} \label{eq:lip} 
|\coll(v,v_*,\theta) - \coll(w,w_*,\theta)| &\leq L_\coll(1 + |\theta|)\left(|v - w| + |v_* - w_*| \right)\\
|\coll(v,v_*,\theta)|& \leq  C_\coll(1 + |\theta|)\left(|v| + |v_*|\right)\,.
\label{eq:lin} 
\end{align}
\end{assumption}

We give in the following an overview of some kinetic models this assumption includes.

\begin{enumerate}[start=1,label={M\arabic*)}]
\item \label{item:M1} Kac's model^^>\cite{kac1956}, where $\D = \R$, 
\[\coll(v,v_*,\theta)= v\cos \theta - v_* \sin \theta\,,\]
and $\theta$ is uniformly sampled from $[0,2\pi]$, that is, $\b(\theta) = 1/(2\pi)$. The parameter of the partner particle is set to $\theta_* = -\theta$.
\item \label{item:M2}
 Kinetic models in economics for wealth distribution, as the 
Cordier--Pareschi--Toscani model^^>\cite{cordier2005market} for 
money asset exchanges. Here,  $\D = [0,\infty)$, and the interaction takes the form 
\[
\coll(v,v_*,\theta = \eta)= v - \gamma(v - v_*) + \eta v_*\,,
\]
with  $\gamma\in (0,1/2)$ and $\eta$ is a zero-mean random variable. The parameter $\theta_* = \eta_*$ of the partner colliding particle is sampled independently.
Similar models employing different parameters are Slanina’s model^^>\cite{slanina2004},   Chackraborti and Chackrabarti  model for gambling^^>\cite{chaka2000}, and, for instance, the ones considered in ^^>\cite{pareschi2006self}.

\item  \label{item:M3} Kinetic models of opinion formation^^>\cite{toscani2006opinion} where the domain $\D \in [-1,1]$ represents the spectrum of possible opinions on a topic, and 
\[
\coll(v,v_*,\theta = \eta) = v - \gamma P(v,v_*)(v - v_*) +  D(v,v_*)\eta \,,
\]
where $\gamma>0$ and $\eta$ is sampled from a zero-mean probability distribution with compact support.  Extension to arbitrary dimensions $d$, with $\Omega = [-1,1]^d$ has also been proposed^^>\cite{boudin2010opinion}. 
The functions $P,D$, as well as the parameters are chosen such that the post-collisional particles $v',v'_*$ still belong to the domain $[-1,1]$. In this settings, Assumption \ref{asm:coll} is satisfied provided $P,D$ are Lipschitz continuous. Analogous kinetic models are used for swarming dynamics in a space homogeneous setting with $\D = \R^3$ (see^^>\cite{AlPa, carrillo2010asymptotic}).

\item  \label{item:M3bis} Kinetic models for traffic flows, in space homogeneous settings^^>\cite{tosin2019traffic}. Here, $\Omega = [0,v_{\max{}}]$, and 
\[
\coll(v,v_*, \theta  = \eta) = v + \gamma I(v,v_*;\rho) + D(v;\rho)\eta
\]
where $I,D$ are some interaction functions which depend on the traffic density $\rho>0$. The random variables $\eta, \eta_*$ are independently sampled from a zero-mean distribution. 
Whether $\coll$ satisfies Assumption \ref{asm:coll} depends on the choice of $I$, $D$.%

\item \label{item:M4} 

Space homogeneous Boltzmann's model for rarefied gases for a specific collisional kernel. Here $\D = \R^3$,  and we consider the collision parametrization
\[
\coll(v,v_*, \theta = e) = v + e \langle e,v_* - v \rangle\,,
\]
with $e$ is uniformly distributed over the sphere $\mathbb{S}^2$, and $\theta_* = \theta = e$. This model, sometimes referred in the literature as the Morgenstern model, has been considered in^^>\cite{morgen1955,carlen1992,carlen2008}. It corresponds to the classical Boltzmann homogenous equation with scattering cross section $\sigma(x) = 1/\sqrt{2(1-x)}$, see^^>\cite{carlen2008}.
We refer to Remark \ref{rmk:boltzmann} for more details on the Boltzmann model.

\item \label{item:M5} Kinetic models in optimization ^^>\cite{benfenati2021binary} with $v,v_* \in \D\subset \Rd$ convex, bounded search domain with arbitrary dimension $d\in \mathbb{N}$. Let $\lambda,\sigma>0$ be positive parameters and $\theta$ be a zero-mean $d$-dimensional random variable, the binary interaction is given by
\[
\coll(v,v_*,\theta) = \Pi_{\D}\left[v + \lambda(v_{\beta}(v,v_*) - v) + \sigma (v_{\beta}(v,v_*) - v) \odot \theta \right] \]
where, $\Pi_\D[\cdot]$ is $\ell_2$-projection into $\D$ and, given the objective function $\E$ and $\beta>0$, $v_{\beta,\E}(v,v_*)$ is a weighted average 
$v_{\beta,\E}(v,v_*) := (e^{-\beta \E(v)}v +e^{-\beta \E(v_*)}v_* )/(e^{-\beta \E(v)}+ e^{-\beta \E(v)})$.
Under the assumption of the objective $\E$ being locally Lipschitz and $\D$ bounded, the weighted averages $v_{\beta, \E}(v,v_*)$ are also locally Lipschitz, see^^>\cite[Lemma 3.2]{carrillo2018analytical}. 
It follows that $\coll$ is also bi-Lipsichitz, while the linear growth condition \eqref{eq:lin} follows from $|v_{\beta, \E}(v,v_*)|\leq|v| + |v_*|$.
\end{enumerate}

\begin{remark} \label{rmk:boltzmann}
Whether the binary collision in the classical Boltzmann equation for dilute gases satisfies Assumption \ref{asm:coll} or not, depends on the parameterization used. For instance, the one used in \ref{item:M4} is bi-Lipschitz, while the one given by 
\begin{equation*}
\label{eq:boltzcoll}
\begin{split}
v' &= \frac{v + v_*}2 + \frac{|v + v_*|}2 \omega \\
v_*' &= \frac{v + v_*}2 - \frac{|v + v_*|}2 \omega \\
\end{split} \qquad \textup{for} \;\; \omega \in \R^3,\, |\omega| = 1
\end{equation*}
is not Lipschitz continuous in the variables $v,v_*\in \R^3$. Though, even with the parametrization \ref{item:M4},  it is not possible to directly apply Theorem \ref{t:nanbu} for arbitrary collisional kernels, as such parametrization makes $\b$ depend on the pre-collisional velocity $v,v_*$, that is, $\b = \b(v,v_*,e)$, also in the case of Maxwellian molecules, unless the specific scattering cross section mentioned in \ref{item:M4} is chosen.
The same situation applies to the Boltzmann equation for granular gases \cite{carrillo2005granular,wagner2007splitting}.

The line of work on DSMC methods typically considers simplified models with collision parameterized as in \ref{item:M4} and
Lipschitz kernels satisfying
\[
| \b(v,v_*,e) - \b(w,w_*,e)| \lesssim |v - v_*| + |w - w_*| 
\]
see^^>\cite[Theorem 4.1]{wagner1992bird} and^^>\cite[Section 3.4.2]{wagner2005book}. On the contrary, in^^>\cite{fournier2016,cortez2018} the analysis covers the more complex cases of hard potential and Maxwell molecules.
\end{remark}

\subsection{Convergence of Monte Carlo approximations}

Consider i.i.d. particles $W^i_n, i = 1,\dots,N$ which are exactly distributed according to the forward Euler iterate $f_n$,  and their associated empirical measure $\overline{f}_n^N= (1/N)\sum_i \delta_{W_n^i}$.  An upper bound of the Monte Carlo error for the Wasserstein-1 distance $\mathbb{W}_1$ (or, equivalenty, in terms of Kantorovich–Rubinstein norm $\|\cdot\|_{\KR}$) was derived in^^>\cite{fournier2015rate}. It holds
\begin{equation} \label{eq:iid}
\mathbb{E}\left\| f_n- {f}_n^N\right\|_\KR  \lesssim M_q^{1/q}(f_n)  \ve(N)
\end{equation}
with $M_q(f_n)$ being the $q$-th moment of $f_n$, and 
\begin{equation} \label{eq:eps1}
\ve(N):=
 \begin{cases}
N^{-1/2}  & \textup{if} \;\;d = 1 \;\;\;\textup{and} \;\; q> 2, \\
N^{-1/2}\log(1+N)  & \textup{if} \;\;d = 2 \;\;\;\textup{and} \;\; q > 2, \\
N^{-1/d} & \textup{if} \;\;d>2 \;\;\;\textup{and} \;\; q > d/(d-1)\,.
\end{cases}
\end{equation}
The constant hidden in \eqref{eq:iid} may depend only on $d$ and $q$ (see Theorem \ref{t:emp} below for the precise statement). It is important to note that the error in  \eqref{eq:iid} is given in terms of expected value because the empirical measure is a random variable since it depends on the $f_n$-distributed random variables $W^i_n$, $i = 1,\dots,N$.

The following error estimate shows that the Nanbu Monte Carlo method (Algorithm \ref{alg:nanbu}) is as efficient as a system of true i.i.d. particles with law $f_n$, which cannot be simulated due to the quadratic collisional operator.  


\begin{theorem}[Convergence of Nanbu method]
Let $f_0 \in \mathcal{P}_q(\Rd)$, with $q>2$ for $d \leq 2$ and $q> d/(d-1)$ for $d>2$,  and let $\b$ be a kernel such that $\b\in \mathcal{P}_{\infty}(\Theta)$.  Consider $f_n$ to be the forward Euler discretization \eqref{eq:euler} for $\Delta t\in(0,1]$ to the Boltzmann-like equation \eqref{eq:boltzmann}, and let $V_n^1, \dots, V_n^N$ be the Nanbu particle system defined by Algorithm \ref{alg:nanbu}, with $f_n^N$ being the corresponding empirical measure. 

For a given time horizon $T>0$, it holds
\[
\sup_{n, n\Delta t\in[0,T]} \mathbb{E} \left\|f_n -  f_n^N\right\|_\KR\leq  C \left(1 + Te^{T \tilde C}\right) M_q^{1/q}(f_0) \ve(N)
\] \label{t:nanbu}
with $M_q(f_0) = \int |v|^q f_0(v)\d v$,  $\ve(N)$ given by \eqref{eq:eps1}, and $C, \tilde{C}$ positive constants that depend only on $q,L_\coll, C_\coll, M_\infty(\b)$.
\end{theorem}
The proof we provide (Section \ref{sec:nanbu}) is based on an accurate coupling of the Nanbu system with an i.i.d. system of particles following the strategy proposed in^^>\cite{fournier2016}.

\subsection{Convergence of forward Euler and full error analysis}

We complement the error introduced by the Monte Carlo strategies with the error introduced by the forward Euler scheme \eqref{eq:euler}, which is of order $\Delta t$, as expected. The following theorem also includes a well-posedness result for the Cauchy problem \eqref{eq:boltzmann}.

\begin{theorem}[Convergence of forward Euler method] \label{t:euler} Let $(f_n)_{n\in \mathbb{N}}$ be constructed according to the  the explicit Euler iteration \eqref{eq:euler} with initial data $f_0\in \mathcal{P}_1(\Rd)$ and step size $\Delta t \in (0,1)$. For a given time horizon $T>0$, consider the interpolation $f^{\Delta t} \in \mathcal{C}([0,T], \mathcal{P}_1(\Rd))$, $f^{\Delta t}(t) := (1 - s)f_n + sf_{n+1}$ with $s = (t - n\Delta t)/\Delta t $ for $t \in [n\Delta t , (n+1)\Delta t]$.

Under Assumption \ref{asm:coll} and $\b\!\in\! \mathcal{P}_\infty(\Theta)$, $f^{\Delta t}$ converges to $f\!\in\! \textup{Lip}([0,T], \mathcal{P}_1(\Rd))$ as $\Delta t \to 0$, which is the unique measure solution (see Definition \ref{def:weaksol}) to \eqref{eq:boltzmann} over the time horizon $[0,T]$, and with initial data $f_0$. Moreover, it holds
\begin{equation}
\sup_{t\in [0,T]} \| f^{\Delta t }(t) - f(t) \|_{\KR} \leq C_{\textup{FE}} \Delta t
\end{equation}
for some positive constant $C_{\textup{FE}}$ which depends on $T,q,L_\coll,C_\coll, M_\infty(\b)$.
\end{theorem}

The constructive proof follows standard techniques of well-posedness for weak measure solutions, see, for instance, ^^>\cite{piccoli2019signed}. Thanks to the simple structure of the kinetic model \eqref{eq:boltzmann} of interest, we conjecture that the existence result can be extended to strong measure solutions, or even strong density solutions ^^>\cite{lu2012measure}.

Finally, we combine the above results to obtain a full error analysis.
\begin{corollary}[Full error analysis] \label{c:full}
Let $f_0 \in \mathcal{P}_q(\Rd)$, with $q>2$ for $d \leq 2$ and $q > d/(d-1)$ for $d>2$,  and let $\b$ be a kernel such that $\b\in \mathcal{P}_{\infty}(\Theta)$. Construct $V_n^1, \dots, V_n^N$ with the Nanbu Algorithm \ref{alg:nanbu} with time step $\Delta t \in(0,1)$, $N$ particles, and initial data $f_0$. 

Under Assumption \ref{asm:coll}, let $f_n^N$ being the corresponding empirical measure and $f \in \textup{Lip}([0,T],\mathcal{P}_1(\Rd))$ be the unique measure solution to \eqref{eq:boltzmann} over $[0,T]$ with $f(0) = f_0$. It holds 
\begin{equation}
\sup_{t_n = n\Delta t \in [0,T]}\mathbb{E}\left\| f^N_{n} -  f(t_n)  \right\|_\KR \leq C_{\textup{FE}} \Delta t  + C_{\textup{MC}}\, \varepsilon(N)
\end{equation}
with $C_{\textup{FE}},C_{\textup{MC}}>0$ constants depending on $T,q,L_\coll, C_\coll, M_\infty(\b), M_q(f_0)$, and $\ve(N)$ given by \eqref{eq:eps1}.
\end{corollary}

\subsection{Extension to Time Relaxed Monte Carlo methods}

It is common in kinetic models to encounter different time scales between macroscopic and microscopic effects. 
When collisions occur at a high rate, a full kinetic treatment becomes computationally expensive due to the large separation of time scales. Moreover, it is often unnecessary, as macroscopic quantities can be accurately described by the system’s asymptotic configuration.

The paradigmatic example is a gas near thermodynamical equilibrium, where the relevant time scale is determined by the Knudsen number—the ratio of the molecular mean free path to the characteristic length scale of macroscopic variations. For large Knudsen numbers, the evolution of the system is well described by the Boltzmann kinetic model, while for small Knudsen numbers the distribution function is close to the asymptotic configuration given by the local Maxwellian.

Consider the rescaled collisional dynamics with scale parameter $\epsilon>0$
\begin{equation} \label{eq:botlzmanneps}
\frac{\partial f}{\partial t} (v,t) = \frac1\epsilon Q(f,f)(v,t)\,.
\end{equation}
When applying the forward Euler scheme \eqref{eq:euler} $f_{n+1} = ( 1- \Delta t/ \epsilon ) f_n + \Delta t/\epsilon Q^+(f_n,f_n)$ one loses the probabilistic Nanbu interpretation of the update unless the restrictive condition $\Delta t \in (0, \epsilon)$ is satisfied. To overcome this situation, the authors in^^>\cite{gabetta1997,pareschi2001trmc} proposed a novel class of Monte Carlo schemes, the Time Relaxed Monte Carlo (TRMC) methods which do not require the condition on $\Delta t $, and are able to capture the asymptotic behaviour of the system, as $\epsilon \to 0$.

Let $f_\infty$ bet the steady associated to \eqref{eq:botlzmanneps}, from which we assume we can draw samples. Consider the parameter $\tau  = 1 - \exp(-\Delta t /\epsilon)$, the first order TRMC method is based on the iterative scheme
\begin{equation} \label{eq:trmc}
f_{n+1} = (1 - \tau)f_n + (1- \tau)\tau Q^+(f_n,f_n)  + \tau^2 f_\infty\,.
\end{equation}
For any $\Delta t, \epsilon>0$, we recover the probabilistic interpretation as $f_{n+1}$ is given by a convex combination of $f_n$, $Q^+(f_n,f_n)$, and $f_\infty$. The TRMC method illustrated in Algorithm \ref{alg:trmc} exploits this, together with the Nanbu strategy to sample from $Q^+(f_n,f_n)$. Note that \eqref{eq:trmc} is asymptotic preserving, as $f_{n+1}$ is relaxed to $f_\infty$ as $\epsilon \to 0$ for fixed $\Delta t$. 

As for the Nanbu algorithm, the TRMC method is as efficient as a system of i.i.d. particles in approximating the correspondent time discrete dynamics \eqref{eq:trmc}.

\begin{theorem}
[Convergence of first order TRMC method]
Let $f_0 \in \mathcal{P}_q(\Rd)$, with $q>2$ for $d \leq 2$ and $q> d/(d-1)$ for $d>2$,  and let $\b$ be a kernel such that $\b\in \mathcal{P}_{\infty}(\Theta)$.  Consider $f_n$ to be the  discretization \eqref{eq:trmc} for $\Delta t\in(0,1]$ to equation \eqref{eq:botlzmanneps} with $f_\infty\in\mathcal{P}_q(\Rd)$, and let $V_n^1, \dots, V_n^N$ be the TRMC particle system defined by Algorithm \ref{alg:trmc}, with $f_n^N$ being the corresponding empirical measure. 

For a given time horizon $T>0$, it holds
\[
\sup_{n, n\Delta t\in[0,T]} \mathbb{E} \left\|f_n -  f_n^N\right\|_\KR\leq  C \left(1 + Te^{T \tilde C}\right) M_q^{1/q}(f_0) \ve(N)
\]\label{t:trmc} 
with $M_q(f_0) = \int |v|^q f_0(v)\d v$,  $\ve(N)$ given by \eqref{eq:eps1}, and $C, \tilde{C}$ positive constants that depend only on $q,L_\coll, C_\coll, M_\infty(\b)$.
\end{theorem}

\begin{algorithm}[t]
\caption{First order Time Relaxed Monte Carlo method}
\label{alg:trmc}
\begin{algorithmic}
\STATE{Sample $N$ particles $V^i_0$, $i = 1, \dots,N$ from $f_0$}
\STATE{$n = 0$}
\WHILE{$n\Delta t < T$}
\STATE{with probability $1 - \tau(\Delta t):$}
\STATE{\hspace{1cm}  $V_{n+1}^i = V_{n}^i$ }
\STATE{with probability $(1-\tau(\Delta t))\tau(\Delta t):$}
\STATE{\hspace{1cm}  select a random particle $j$}
\STATE{\hspace{1cm}  sample a parameter $\theta$ from $\b(\theta)$}
\STATE{\hspace{1cm}  $V^i_{n+1} =\coll(V_n^i, V_n^j, \theta)$}
\STATE{with probability $\tau(\Delta t)^2:$}
\STATE{\hspace{1cm} sample a particle $v$ from equilibrium $f_\infty$}\STATE{\hspace{1cm}  $V_{n+1}^i = v$}
\STATE{$n = n+1$}
\ENDWHILE
\end{algorithmic}
\end{algorithm}

\begin{remark}
Scheme \eqref{eq:trmc} represents only a sub-class of first order TRMC methods. TRMC schemes were derived^^>\cite{gabetta1997,pareschi2001trmc,caflisch1999trmc} for the Boltzmann equation starting from the solution representation via Wild's sums^^>\cite{wild1951,carlen2000}. The general high-order scheme of order $m\geq 1$ takes the form 
\begin{equation*}
f_{n+1} = \sum_{k=0}^m A_k(\tau) f^k_n + A_{m+1}(\tau)f_\infty
\end{equation*}
with $f^n_k$ recursively defined as $f_n^{k+1} = \sum_{h=0}^k Q^+(f_n^h,f^{k-h}_h)/(k+1)$, and $A_k(\tau), k =0,1,\dots,m+1$ positive weights satisfying certain admissibility criteria  as $\tau \to0$ and $\tau \to 1$ (see^^>\cite[Proposition 3.2]{pareschi2001trmc}). Generalizations where the weights have been computed using Runge-Kutta methods have been presented in^^>\cite{Dimarco2011}. We conjecture that convergence results of type \eqref{t:trmc} can be extended to these higher order TRMC methods by iterating the same argument of the proof to the higher order terms $f_n^h$, $h= 2, \dots, m$. 
\end{remark}

\begin{remark} TRMC methods belong to the class of Asymptotic Preserving (AP) schemes^^>\cite{jin2010ap}, which remain efficient even in the asymptotic regime $\epsilon \to 0$. A comprehensive error analysis, including the dependence on the time step $\Delta t$ and parameter $\epsilon>0$, must account for the convergence rate to the asymptotic state $f_{\infty}$ as $\epsilon \to 0$ and $t \to \infty$. This convergence behavior is generally model-dependent. For a detailed framework for deriving such error estimates, we refer the reader to^^>\cite{jin2010ap}.
\end{remark}

\section{Preliminaries and notation}
\label{sec:prel}

We indicate with  $\mathcal{P}(\Rd)$ the set of Borel probability measures over $\Rd$, and with $M_p(f) := (\int |v|^p f)^{1/p}$, $p \in (1,\infty)$, the $p$-th moment of $f\in \mathcal{P}(\Rd)$. For $p =\infty$, we set $M_{\infty}(f) = \sup_{v\in \supp(f)} |v|$. The set of probability measures with bounded moments up to $p$ is denoted with $\mathcal{P}_p(\Rd)$, and $\mathcal{P}^{ac}_p(\Rd) \subset \mathcal{P}_p(\Rd)$ is the one that only includes probability measures absolutely continuous with respect to Lebesgue. We will sometime abuse the notation and indicate the density of $f\in\mathcal{P}_p^{ac}(\Rd)$ again with $f$. For a measurable function $\phi: \Rd \to \Rm$ and $f\in \mathcal{P}(\Rd)$, $\phi_\#f \in \mathcal{P}(\Rm)$ is the push-forward measures defined by $\phi_\#f (A) = f(\phi^{-1}(A))$ for any open set $A\subset \Rm$. For any bounded measurable set $A \subset \Rd$ we denote with $|A|$ its Lebesgue measure. With $\mathcal{C}_0^\infty(\Rd)$ we denote the set of smooth and compactly supported test function, while $\textup{Lip}_1(\Rd)$ is the set of Lipschitz functions with Lipschitz constant $\|\phi \|_{\textup{Lip}} \leq 1$. Given a test function $\phi$ and a Radon measure $\mu$ defined on the same space, we will sometimes use the compact notation  $\langle \phi,\mu\rangle := \int\phi(v) \mu(\d v)$.

If not specified, random variables are taken from an abstract probability space $(\tilde \Omega, \mathcal{F}, \mathbb{P})$. Following^^>\cite{tanaka1978,fournier2016}, we will sometimes use an auxiliary probability space given, for instance, by $[0,N)$, for some $N\in \mathbb{N}$, with the Borel $\sigma$-algebra and normalized Lebesgue measure. We call random variables defined on this auxiliary space $\alpha$-random variables.

We say $X\sim f$ for $f\in\mathcal{P}(\Rd)$ if the law of the random variable is $f$, and sometimes write $\law(f) = X$. With $\unif(A)$, we indicate the uniform probability measure over a bounded measurable set $A\subset\Rd$, 
and $\bern(\tau)$, $\tau \in [0,1]$ is the Bernoulli distribution, that is, $\bern(\tau) = (1-\tau)\delta_0 + \tau \delta_1$, $\delta_x$ being the Dirac delta probability measure centered in $x\in \Rd$. For $A\in \mathcal{F}$, $\mathbf{1}_A$  is the indicator function $\mathbf{1}_A(\omega) = 1$ if $\omega\in A$, and $\mathbf{1}_A(\omega) = 0$ otherwise. 

Given $f, g \in \mathcal{P}_p(\Rd)$, we consider the Wasserstein distance with exponent $p\geq 1$ 
\begin{equation}\label{eq:wass}
\W_p(f,g) := \left(\min_{\gamma \in \Gamma(f, g)} \int |v - w|^p \;\gamma(\d v, \d w) \right)^{1/p}
\end{equation}
where $\Gamma(f,g)$ is the set of transport plans between $f$ and $g$.  We recall that $\W_p(\cdot,\cdot)$ metrizes weak convergence in duality with continuous bounded functions, and that, for $p = 1$, the dual formulation reads
\begin{equation} \label{eq:wass_dual}
\W_1(f,g) =  \max \left \{ \int \phi (v) f(\d v) -  \int \phi (w) g(\d w) \;:\; \phi \in \textup{Lip}_1(\Rd)  \right\}\,.
\end{equation}
If not stated differently, we consider $p = 1$ and indicate with $\Gamma_o(f,g)$ the set of couplings that are optimal with respect to the $\ell_1$ cost $|v-w|$. We refer to the book^^>\cite{villani2009} for more details on Wasserstein distances and their properties.
From the dual formulation \eqref{eq:wass_dual}, we can see that $\W_1(f,g)  = \|f - g\|_\KR$ where $\|\cdot\|_{\KR}$ is the previously introduced Kantorovich--Rubinstein norm 
$
\|  \mu \|_\KR:= \sup \left \{ \int_{\Rd} \phi(v)\mu(\d v)\, :\, \phi \in \textup{Lip}_1(\Rd)  \right\},
$
for any signed Radon measure $\mu\in\mathcal{M}(\Rd)$.

Consider $\mu \in \mathcal{P}(\Rd)$ and its empirical approximation $\mu \in \mathcal{P}(\mathcal{P}(\Rd))$ given by
$
\mu^N = 1/N\sum_{i=1}^N \delta_{W^i}$
with
$W^i \sim f$ i.i.d.$\,$.
We recall some error bounds in terms of Wasserstein distances.
\begin{theorem}[{\cite[Theorem 1]{fournier2015rate}}] \label{t:emp}
Let $\mu \in \mathcal{P}(\Rd)$ and let $p>0$. Assume that $M_q(\mu)<\infty$ for some $q>p$.  There exists a constant $C$ depending only on $p,d,q$ such that for all $N \geq 1$:
$
\mathbb{E}\left[  \W_p(\mu,\mu^N)\right] \leq C M_q^{p/q}(\mu) \ve_{p}(N)
$
with 
\begin{equation*}
\ve_{p}(N):=
\begin{cases}
N^{-1/2} + N^{-(q-p)/q} & \textup{if} \;\;p>d/2 \;\;\;\textup{and} \;\; q \neq 2p, \\
N^{-1/2}\log(1+N) + N^{-(q-p)/q} & \textup{if} \;\;p=d/2 \;\;\;\textup{and} \;\; q \neq 2p, \\
N^{-p/d} + N^{-(q-p)/q} & \textup{if} \;\;p\in(0,d/2) \;\;\;\textup{and} \;\; q \neq d/(d-p).
\end{cases}
\end{equation*}

\end{theorem}

Note that the order of convergence given by \eqref{eq:eps1} is a consequence of the above theorem with $p=1$.
Also, it tells us that the error introduced by any Monte Carlo strategy is related to the moments of the kinetic density. Therefore, we provide an exponential bound on the $q$-th moment of the Euler discretization $f_n$ relying only on Assumption \ref{asm:coll}. 

\begin{lemma}[Moments estimate]\label{l:moment}
Assume $f_{0} \in \mathcal{P}_q(\Rd)$, for $q \geq 1$ and $\b \in \mathcal{P}_\infty(\Theta)$, and let $f_n$ be the forward Euler discretization defined by \eqref{eq:euler}.  If the collision maps $\coll$ satisfies Assumption \ref{asm:coll}, then 
\begin{equation}
M_q^{1/q}(f_n) \leq e^{Cn\Delta t} M_q^{1/q}(f_0)
\end{equation}
with $C>0$ a constant depending only on $q, C_\coll, M_{\infty}(b)$.
\end{lemma}
\begin{proof}
From Assumption \ref{asm:coll}, in particular \eqref{eq:lin}, we have
\begin{equation*} 
\int |v'|^q \b(\d \theta)
 \leq \int C_\coll^q(1 + |\theta|)^q\left( |v| + |v_*|\right)^q \b(\d \theta)
 \leq C (1 + M^q_\infty(\b)) \left( |v|^q + |v_*|^q\right) 
\label{eq:collest}
\end{equation*}
for some positive constant $C_0 = C_0(C_\coll,q)$, where we recall $M_\infty(\b) = \sup \{ |\theta|\, : \,\theta\in \supp(\Theta)\}$. By definition of the collisional operator $Q$, and its positive component $Q^+$, by applying the above estimate we obtain
\begin{align*}
\int |v|^q Q^+(f_n,f_n)(\d v) =
\frac12\! \iiint \!\left( |v'|^q + |v_*'|^q  \right) f_n(\d v)f_n(\d v_*)\b(\d \theta)   \leq  C_1^q \int |v|^q f_n(\d v),
\end{align*}
for some $C_1 = C_1(q,C_\coll,M_\infty(\b))>0$. For the Euler update \eqref{eq:euler}, therefore,  it holds
\begin{align*}
M_q^{1/q}(f_{n+1})& \leq (1- \Delta t)M^{1/q}_q(f_n) + \Delta t M_q^{1/q}\left(Q^+(f_n,f_n) \right),\\
& \leq (1 + (C_1 - 1)\Delta t) M_q^{1/q}(f_n)\,.
\end{align*}
By iterating the estimate at all time steps, we obtain for $C = C_1 - 1$
\[
M_q^{1/q}(f_n) \leq (1 + \Delta t C)^n M_q^{1/q}(f_0)\,.
\]
and conclude by noting that $1 + \Delta t C \leq e^{C\Delta t}$. 
\end{proof}
\begin{remark}
The above estimate of the $q$-th moment is clearly not sharp, as it is intended to cover the wide class of models satisfying Assumption \ref{asm:coll}. For a given model, sharper estimate can typically be derived by exploiting the particular structure of the collisions, as done, for instance, in models for wealth dynamics^^>\cite{pareschi2006self}, or the Boltzmann interaction \ref{item:M4} where the second moment, the energy, is conserved by the collision.
\end{remark}

We end the section by defining a notion of solution to the Cauchy problem \eqref{eq:boltzmann}.

\begin{definition}[Weak measure solution to \eqref{eq:boltzmann}]  Let $T>0$ be a time horizon, $\b \in \mathcal{P}(\Theta)$, and $f_0\in \mathcal{P}(\Rd)$ an initial datum. We say $f\in \mathcal{C}([0,T],\mathcal{P}(\Rd))$ is a weak measure solution to the Cauchy problem \eqref{eq:boltzmann}  over the interval $[0,T]$, if $f(0) =f_0$ and
for any test function $\phi\in \mathcal{C}_0^\infty(\Rd)$ and almost every $t\in[0,T]$ it holds
\begin{equation*}
\frac{\d}{\d t}\int \phi(v) f(t,\d v) = \frac12 \iiint \b(\d \theta) \left(\phi(v') + \phi(v'_*) - \phi(v) - \phi(v_*) \right) f(t,\d v) f(t, \d v_*) \,.
\end{equation*}
\label{def:weaksol}
\end{definition}

\section{Error analysis of Nanbu Monte Carlo method}
\label{sec:nanbu}

\subsection{The Nanbu particle system} \label{sec:nanbu:particle}

We will consider a particle system $\BV = (V^1, \dots,V^N)$ to be a random variable taking values in $\R^{Nd}$. The Nanbu particles generated with Algorithm \ref{alg:nanbu} can be seen as a realization of a Markov process $(\BV_n)_{n\in \mathbb{N}}$ constructed in the following way. First, we independently sample $N$ independent particles
\[
V_{0}^i \sim f_{0}\,, \qquad i = 1, \dots N\,.
\]
and fix a time step $\Delta t \in(0,1]$. Recall $\b\in \mathcal{P}_\infty(\Theta)$ is the probability distribution of the collision parameter, where $\Theta$ is some given parameters space. At every step $n = 1,2, \dots, $ we consider for each particle $i=1,\dots,N$ three random variables $\tau^i_n, \alpha^i_n, \theta^i_n$ with
$
 \tau_n^i \sim \textup{Bern}(\Delta t)$,
$\alpha_n^i \sim \unif[0,N)$,
and
$\theta_n^i \sim \b$.
The role of $\tau^i_n$ is to determine if the $i$-th particle collides or not at step $n$. If the particle collides, the partner particle is determined by the variable $\alpha^i_n$ through 
$
\j(\alpha):= \lfloor \alpha \rfloor + 1\,,
$
where $\lfloor\cdot \rfloor$ is the floor map. The third variable $\theta^i_n$ indicates the parameter of the collision. 

In this way, we can write the particles update of the Nanbu method as
\begin{equation} \label{eq:syst:nanbu}
V_{n+1}^i = (1-\tau_n^i)V_n^i + \tau_n^i \coll\left(V_n^i, V_n^{\j(\alpha_n^i)}, \theta_n^i \right) \qquad i = 1,\dots,N\,.
\end{equation}
As described by Algorithm \ref{alg:nanbu}, with probability $1-\Delta t$ particle $i$ does not collide, while with probability $\Delta t$ it does. If it does, particle $i$ collides with particle $\j(\alpha_n^i)$, which is uniformly chosen among the $N$ particles. We note that it may happen that particle $i$ collides with itself. 

\subsection{The nonlinear particle system}

To prove convergence of the Nanbu particle system $\BV_n$ towards the Euler approximation $f_n \in \mathcal{P}(\Rd)$ of the kinetic equation \eqref{eq:boltzmann} we will consider an intermediate approximation of $f_n$ made of a nonlinear particle system $(\BW_n)_{n\in \mathbb{N}}$,  $\BW_n = (W_n^1, \dots,W_n^N)$,  of i.i.d. particles such that 
\[
W_{n}^i \sim f_{n}\,, \qquad i = 1, \dots N\,, \qquad \textup{for all}\quad n \geq 0\,.
\]
Let $\overline{f}^N_n = (1/N)\sum_i \delta_{W_n^i}$ be the corresponding empirical measure, our objective to is estimate from above the Nanbu approximation error (in terms of Wasserstein-1 distance) via triangular inequality:
\[
\W_1\left(f_n,f^N_n\right) \leq \W_1\left(f_n,\overline{f}_n^N\right) + \W_1\left(\overline{f}_n^N, f^N_n\right)\,.
\]
We note that the first term of right-hand side can be bounded via Theorem \ref{t:emp} since $W^i_n$ are i.i.d. and $f_n$-distributed. In the following, we show to how to build the auxiliary system $(\BW_n)_{n\in \mathbb{N}}$ such that the second term can be controlled.

At time step $n = 0$, $\BW_0$ is a copy of $\BV_0$:
$W_{0}^i = V^i_{0}$, $i = 1,\dots,N\,.$
For each $n=0,1,\dots$ we consider a $f_n$-distributed $\alpha$-random variable $W^*_{n}$, that is, 
\[
W^*_{n}(\alpha) \sim f_n \quad \textup{if} \quad \alpha \sim \unif[0,N)\,.
\]
Intuitively, the role of this auxiliary variable is to provide a colliding partner to each of the particles $W^i_{n}$, $i = 1,\dots,N$.
An explicit choice of $W^*_{n}$ will be given later in Lemma \ref{l:coupling}, and this is not relevant at the moment. We correlate the two particles systems by using, in the update of $W^i_n$, the same random variables $\tau^i_n$, $\alpha_n^i$, and $\theta^i_n$ used for the update of $V^i_n$:
\begin{equation} \label{eq:syst:non}
W^i_{n+1} = (1-\tau_n^i)W_n^i + \tau_n^i \coll\left(W_n^i, W_{n}^{*}(\alpha_n^i), \theta_n^i\right)  \qquad i = 1,\dots,N\,.
\end{equation}

Since the random variables used in the above update are all independent with each other, $W^i_n$, $i = 1,\dots,N$ are also independent. We now check that if $W^i_n \sim f_n$ then $W^i_{n+1}\sim f_{n+1}$. Recall from the definition \eqref{eq:operator} of $Q$ that the gain part of the collisional operator is defined by
\begin{equation}
\int \phi(v) Q^+(g,g)(\d v)= \frac12\iiint \left( \phi(v') + \phi(v'_*) \right) \b(\d \theta)\, g(\d v)\,g(\d v_*)
\end{equation}
for all $\phi\in \mathcal{C}_b(\Rd)$ and $g\in \mathcal{P}(\Rd)$. By exploiting the symmetry of the collisional dynamics $v' = \coll(v,v_*,\theta)$ and $v'_* = \coll(v_*,v,\theta_*)$, where both $\theta,\theta_*$ are $\b$-distributed, we note that the gain operator can be defined compactly as $Q(g,g): = \coll_\#(g\otimes g \otimes \b)$.

Assuming $W^i_n \sim f_n$, and since $\mathbb{E}[\tau^i_n] = \Delta t$, from the update \eqref{eq:syst:non} we have
\begin{align*}
\mathbb{E} \phi(W^i_{n+1}) & = (1-\Delta t)\mathbb{E}\phi(W_n^i) + 
\Delta t \mathbb{E} \coll(W_n^i, W_{n}^*(\alpha_n^i), \theta_n^i) \\
& = (1-\Delta t) \int \phi(v) f_n(\d v) + 
\Delta t \iiint  \phi\left(\coll(v,v_*,\theta)\right) b (\d \theta)f_n(\d v) f_n(\d v_*) \\
& = (1-\Delta t) \int \phi(v) f_n(\d v) + 
\Delta t \int \phi(v) Q^+(f_n,f_n)(\d v) \,.
\end{align*}
The last expression corresponds exactly to the forward Euler update \eqref{eq:euler} tested against $\phi$, and therefore we can conclude that $W^i_{n+1}$ is $f_{n+1}$-distributed.

We remark that the particle system $\BW = (W^1,\dots, W^N)$ is nonlinear as the particles collide at every $n\geq 1$ with the $\alpha$-random variable $W^*_{n}$, whose definition depends on the law $f_n$. This is also the reason why $\BW$ cannot be numerically simulated, unlike the Nanbu particles system $\BV$. We refer to^^>\cite[Appendix A.4]{diez2022review1} for a remainder on non-linear Markov processes.

\subsection{Coupling and proof of Theorem \ref{t:nanbu}}

We have seen already that the Nanbu particle system \eqref{eq:syst:nanbu} and the nonlinear particle system \eqref{eq:syst:non} are coupled by the initial conditions $V^i_0 = W_0^i$, and by the random variables $\tau_n^i, \alpha^i_n, \theta^i_n$, $i = 1,\dots,N$. We now provide a way of constructing the auxiliary $\alpha$-random variable $W^*_{n}(\cdot)$ in a way such that the particle system $V^{\j{(\cdot)}}$ of colliding particles (seen as an $\alpha$-random variable) is optimally coupled with the colliding particle $W^*_{n}(\cdot)$ of the nonlinear system. Optimality is intended with respect to the Wasserstein-1 distance. 

\begin{lemma}[Coupling for Nanbu method.] \label{l:coupling} Consider $f\in\mathcal{P}_1(\Rd)$, and $\v = (v^i, \dots, v^N)\in (\Rd)^N$ with $\mu_\v = (1/N)\sum_i \delta_{v^i}$. There exists a measurable mapping 
\begin{align*}
W_f^*:(\Rd)^N \times [0,N) &\to \Rd\\
(\v,\alpha) &\mapsto W^*(\v,\alpha)
\end{align*}
with the following property: if $\alpha$ is uniformly chosen from $[0,N)$, then the pair $(W_f^*(\v,\alpha),v^{\j(\alpha)})$ is an optimal coupling between $f$ and $\mu_{\v}$.
\end{lemma}
\begin{proof} The proof is strategy follows the technique introduced in^^>\cite[Lemma 3]{cortez2016}. 
Let $\pi_{\v} \in\mathcal{P}(\Rd \times \Rd)$ be the optimal transference plan between $f$ and $\mu_\v$. Thanks to a measurable selection result, see, for instance,^^>\cite[Corollary 5.22]{villani2009}, there exists a measurable mapping
\[
\v \mapsto \pi_{\mu_\v} \quad \textup{s.t.} \quad \pi_{\mu_\v} \in \Gamma_o(f,\mu_\v)
\]
where, we recall that $\Gamma_o(f,\mu_\v)$ is the set of optimal mappings from $f$ to $\mu_\v$.
Define for any Borel set $A\subseteq \Rd$
\[G^i(\v,A) := \frac{\pi_{\mu_\v}(A \times \{v^i\})}{\pi_{\mu_\v}(\Rd \times \{v^i\})}\,.\]
We note that $G^i$ is a probability kernel from $(\Rd)^N$ into $\Rd$, thanks to the measurability of $\v \mapsto \pi_{\mu_\v}$, and so there exists $g^i = g^i(\v,\beta)$  such that $\law(g^i(\v,\cdot)) = \law(G^i(\v,\cdot))$ if $\beta\in \unif[0,1)$, see^^>\cite[Lemma 4.22]{kallenberg2021}. This procedure is called \textit{randomization} of $G^i(\v,\cdot)$.
Let us define the mapping $W_f^*$ as
\[
W_f^*(\v,\alpha) := \sum_{i=1}^N \bm{1}_{\{\j(\alpha) = i\}} g^i(\v, \alpha - \lfloor \alpha \rfloor)\,.
\]
To conclude, we need to show that $(W_f^*(\v,\cdot),v^{\j(\cdot)})$
have a joint distribution $\pi_{\mu_\v}$ for $\alpha \sim \unif[0,N)$. Take a Borel set $B \subseteq \Rd$ and $j\in\{1,\dots,N\}$, we have indeed
\begin{align*}
\mathbb{P} \left( W_f^*(\v,\alpha) \in A,\, v^{\j(\alpha)} = v^j \right) &= 
\mathbb{P} \left ( W_f^*(\v,\alpha)  \in A\, \middle |\, v^{\j(\alpha)} = v^j\right)\mathbb{P}\left( v^{\j(\alpha)} = v^j \right) \\
& = \mathbb{P} \left ( g^j(\v,\alpha) \in A \right) \frac1N  = \frac{\pi_{\mu_\v}(A \times \{v^j\})}{\pi_{\mu_\v}(\Rd \times \{v^j\})} \frac1N \\
& =  \frac{\pi_{\mu_\v}(A \times \{v^j\})}{1/N} \frac1N   = \pi_{\mu_\v}(A \times \{v^j\})\,.
\end{align*}
\end{proof}

\begin{proof}[Proof of Theorem \ref{t:nanbu}] To study the distance between the Nanbu particle system $\BV_n$ and the nonlinear system $\BW_n$, or, more precisely, between their respective empirical distributions $f_n^N$ and $\overline{f}_n^N$, we couple particle $V_n^i$ with $W_n^i$ for all $i = 1,\dots,N$ and $n \geq 0$. This coupling, in principle, is sub-optimal with respect to the Wasserstein-1 distance
and so it holds
$
\W_1(f^N_n, \overline{f}^N_n) \leq (1/M)\sum_{i=1}^N |V^i_n - W^i_n|\,.
$
To generate the nonlinear particle system $\BW$ we use the map $W_{f_n}^*(\cdot, \alpha)$  constructed in Lemma \ref{l:coupling} by choosing  $W_n^*(\cdot):= W_{f_n}^*(\BV_n,\cdot)$ as auxiliary colliding particle. This is a possible choice as the $ W_{f_n}^*(\BV_n,\cdot)$ is $f_n$-distributed as $\alpha$-random variable.

The two particle systems are, therefore given by $V_0^i = W_0^i$ and
\begin{equation}
\begin{cases}
V_{n+1}^i = (1-\tau_n^i)V_n^i + \tau_n^i \coll \left (V_n^i, V_n^{\j(\alpha_n^i)} ,\theta^i_n\right) & i = 1,\dots,N\, \\
W^i_{n+1} = (1-\tau_n^i)W_n^i + \tau_n^i \coll\left(W_n^i, W_{f_n}^{*}(\BV_{n},\alpha_n^i),\theta^i_n \right)  & i = 1,\dots,N\,.
\end{cases}
\end{equation} 

By using the Lipschitz continuity assumption on $\coll$, it follows
\begin{small}
\begin{align*}
| V_{n+1}^i &- W_{n+1}^i|  = \Big |  (1-\tau_n^i)V_n^i + \tau_n^i \coll\Big(V_n^i, V_n^{\j(\alpha_n^i)}, \theta^i_n \Big) \\
& \quad -   
(1-\tau_n^i)W_n^i - \tau_n^i \coll\Big(W_n^i, W_n^{*}(\BV_{n},\alpha_n^i),\theta^i_n \Big) 
     \Big | \\
    & \leq (1-\tau_n^i) | V_n^i - W_n^i | + \tau_n^i \left| 
     \coll\Big(V_n^i, V_n^{\j(\alpha_n^i)},\theta^i_n\Big)
     \!- \!\coll\Big(W_n^i, W_n^{*}(\BV_{n},\alpha_n^i),\theta^i_n \Big) 
          \right | \\
          &  \leq (1-\tau_n^i) | V_n^i - W_n^i |+ \tau_n^i  L_\coll(1 + |\theta^i_n|) \big (\left | V_n^i - W_n^i\right | + \big |V_n^{\j(\alpha_n^i)} -  W_{f_n}^{*}(\BV_{n},\alpha_n^i)\big | \big).
\end{align*}\end{small}
We take the expectation and obtain, thanks to the optimal choice of $W_n^*$, 
\begin{equation}
\mathbb{E} | V_{n+1}^i - W_{n+1}^i|  \leq (1 + \Delta tC_1 )\mathbb{E} | V_{n}^i - W_{n}^i| + \Delta t C_2 \W_1\left(f^N_n, f_n \right)
\label{eq:est1}
\end{equation}
for some constant $C_1, C_2>0$ which depend only on $q,L_\coll, M_\infty(\b)$.
We note that the second term can be bounded as 
\begin{align}
\mathbb{E} \W_1(f^N_n, f_n ) &\leq \mathbb{E} \W_1(f^N_n, \overline{f}^N_n ) + 
\mathbb{E} \W_1(\overline{f}_n^N, f_n ) \notag\\
&  \leq  \frac1{N}\sum_{i=1}^N \mathbb{E}| V_{n}^i - W_{n}^i| + 
\mathbb{E}\W_1(\overline{f}^N_n, f_n )\,. \label{eq:est2}
\end{align}
By summing \eqref{eq:est1} for all $i = 1,\dots,N$, and dividing by $N$, we obtain
\begin{equation*}
\frac1N\sum_{i=1}^N \mathbb{E}|V^i_{n+1} - W^i_{n+1} | \leq (1 + \Delta t(C_1 + C_2))\frac1N\sum_{i=1}^N \mathbb{E} |V^i_n - W_n^i| +\Delta t C_2C \mathbb{E}\W_1\left(\overline{f}^N_n, f_n \right).
\end{equation*}
Iterating the argument for the time step $0\leq h\leq n$ leads, for some $C_3>0$,
\begin{multline}
\frac1N \sum_{i=1}^N \mathbb{E}|V^i_n - W^i_n|  \leq (1 + \Delta t C_3)^n \frac1N\sum_{i=1}^N \mathbb{E} |V^i_0 - W_0^i| 
\\+\Delta t C_3 \sum_{h=0}^{n-1}(1-\Delta t C_3)^h \mathbb{E}\W_1\left(\overline{f}^N_h, f_h \right)\,.
\label{eq:estA}
\end{multline}
We note that the first term on the right-hand side is zero, thanks to the choice of initial data $V_0^i = W_0^i$, $i = 1,\dots,N$. For the second term, we note that $\overline{f}^N_h$ is the empirical measure associated with the $f_h$-distributed i.i.d. nonlinear particle system $W^i_h, i = 1,\dots,N$. Therefore, we apply Theorem \ref{t:emp}, and Lemma \ref{l:moment} to get 
\begin{equation*}
 \mathbb{E}\W_1\left(\overline{f}^N_h, f_h \right) \leq C M_1^{1/q}(f_h) \ve (N)\leq C_4 e^{Ch\Delta t} M_1^{1/q}(f_0) \ve(N)\,.
\end{equation*}
By plugging this estimate in \eqref{eq:estA} and by using $1+x \leq e^x$, we obtain
\begin{align*}
\frac1N \sum_{i=1}^N \mathbb{E}|V^i_n - &W^i_n|  \leq (1 + \Delta t C_3)^n \frac1N\sum_{i=1}^N \mathbb{E} |V^i_0 - W_0^i| \\
& \qquad+\Delta t C_3 C_4 \sum_{h=0}^{n-1}(1-\Delta t C_3)^h  e^{Ch\Delta t} M_1^{1/q}(f_0) \ve(N) \\
& \leq
e^{C_3 n\Delta t} \frac1N\sum_{i=1}^N \mathbb{E} |V^i_0 - W_0^i| +  C_3 C_4 (\Delta t n) e^{\max\{C_3,C\} n\Delta t}M_1^{1/q}(f_0) \ve(N)\,.
\end{align*}
Since $V^i_0 = W^i_0$, for some $C_5,C_6>0$, it holds
\[
\frac1N \sum_{i=1}^N \mathbb{E}|V^i_n - W^i_n| \leq C_5 T e^{C_6T} M_q^{1/q}(f_0) \ve(N)\,.
\]
To conclude, we use \eqref{eq:est2} again and obtain
\begin{align*}
\mathbb{E} \W_1\left(f^N_n, f_n \right) & \leq  \frac1{N}\sum_{i=1}^N \mathbb{E}| V_{n}^i - W_{n}^i| + 
\mathbb{E}\W_1\left(\overline{f}^N_n, f_n \right) \\
&\leq C_5 T e^{C_6T} M_q^{1/q}(f_0)\ve(N) + C e^{CT}M_q^{1/q}(f_0)\ve(N) \\
& \leq C_7( 1 + T)e^{C_6T} M_q^{1/q}(f_0)\ve(N)\,.
\end{align*}
Constants $C_6,C_7$ depend on $q,L_\coll, C_\coll, M_\infty(\b)$, but are independent on $d,N$.
\end{proof}

\subsection{Extension to Time Relaxed Monte Carlo methods }
\label{sec:trmc}

Recall the first order TRMC method is based on the time discrete scheme
\[f_{n+1} = (1 - \tau) f_n  + \tau(1 - \tau) Q(f_n,f_n) + \tau^2 f_\infty\]
for $\tau \in (0,1)$.
As in Section \ref{sec:nanbu}, we write the particles evolution by auxiliary random variables. For every particle $i$, we consider $\tau_n^{1,i}\sim \bern(\tau)$ and $\tau_n^{2,i}\sim \bern(\tau)$ independent of each other. We also sample $\alpha^i_n\sim \unif[0,N)$, $\theta^i_{n} \sim \b$, and 
\[
M^i_n \sim f_{\infty}
\]
independent for all $i = 1,\dots,N$. The particle system generated by Algorithm \ref{alg:trmc} can then be iteratively defined as 
\begin{equation} \label{eq:syst:trmc}
V_{n+1}^i = (1 - \tau^{i,1}_n)V_n^i + \tau^{i,1}_n(1 -  \tau_n^{i,2})  \coll\left(V_n^i, V_n^{\j(\alpha_n^i)}, \theta_n^i \right) +  \tau^{i,1}_n \tau_n^{i,2} M^i_n\,.
\end{equation}
To define the correspondent non-linear system $\BW_n$, we employ the auxiliary $\alpha$-random variable $W^*_{f_n}(\BV_n, \cdot)$ given by Lemma \ref{l:coupling}. Then, starting from $\BW_0 = \BV_0$, we define
\begin{equation} \label{eq:syst:trmcnon}
W_{n+1}^i = (1 - \tau^{i,1}_n)W_n^i + \tau^{i,1}_n(1 -  \tau_n^{i,2})  \coll\left(W_n^i, W^*_{f_n}(\BV_n, \alpha_n^i), \theta_n^i \right) +  \tau^{i,1}_n \tau_n^{i,2} M^i_n\,.
\end{equation}

\begin{proof}[Proof of Theorem \ref{t:trmc}]
The proof relies on coupling the two particle systems $\BV_n$ and $\BW_n$, and follows the same steps as the proof of Theorem \ref{t:nanbu}. We omit the details for brevity. We only note that the relaxation towards the asymptotic distribution $f_\infty$ does not introduce an additional error, since the systems share the same particles $M^i_n$, $i  = 1,\dots,N$.
In particular, the equivalent of estimate \eqref{eq:est1} is given in this case by 
\begin{align*}
    \mathbb{E} | V_{n+1}^i - W_{n+1}^i| &\leq  (1 + \tau C_1 )\mathbb{E} | V_{n}^i - W_{n}^i| + \tau(1 - \tau) C_2 \W_1\left(f^N_n, f_n \right)\,.
\end{align*}
    
\end{proof}

\section{Error analysis of forward Euler scheme}
\label{sec:euler}

In this section, we provide a proof of the existence and uniqueness of weak measure solutions (Definition \ref{def:weaksol}) to the Cauchy problem \eqref{eq:boltzmann}. We follow standard arguments of a constructive proof, see e.g.^^>\cite{piccoli2019signed}, that is, we first show that the forward Euler approximation forms a Cauchy sequence (Lemma \ref{l:cauchy}), and then that its limit is indeed a solution to the Boltzmann-like equation (Lemma \ref{l:existance}). Finally, we provide a quantitative estimate on the approximation error of the forward Euler scheme.

Recall that the iterative forward Euler scheme is defined as
\[
f_{n+1} = (1 - \Delta t) f_n + \Delta t Q^+(f_n,f_n)\,,
\]
with initial data $f_0\in\mathcal{P}_1(\Rd)$. For a time horizon $T>0$, we consider the interpolation $\fdt \in \mathcal{C}([0,T], \mathcal{P}_1(\Rd))$ given by 
\begin{equation} \label{eq:eulert}
\fdt (t) = (1- (t -  n \Delta t )) f_n  + (t - n \Delta t) Q^+(f_n,f_n) \qquad \textup{for} \;\; t\in [n\Delta t, (n+1)\Delta t )\,.
\end{equation}

\begin{lemma} \label{l:cauchy} Under Assumption \ref{asm:coll} and $\b\in\mathcal{P}_\infty(\Theta)$, let $\Delta t_k = T/2^k$ and $f^{\Delta t_k}$ be defined by \eqref{eq:eulert}. Then $\{f^{\Delta t_k}([0,T])\}_{k\in \mathbb{N}}$ is a Cauchy sequence in $\mathcal{C}([0,T],\mathcal{P}_1(\Rd))$.
\end{lemma}
\begin{proof}
Let $k$ be fixed, and $\Delta t = \Delta t_k$ be the corresponding time step. We compare the error of the two iterates $\fdt$ and $f^{\Delta t/2}$. For a given time $t_n = n\Delta t$, we estimate the Wasserstein-1 distance between the two iterates at $t \in [t_n, t_n + \Delta t ]$. Recall we have for $s\in [0,\Delta t/2]$
\begin{equation*}
\begin{split}
\fdt(t_n + s) &= (1 - s)\fdt(t_n)  + sQ^+(\fdt(t_n), \fdt(t_n)) \\
f^{\Delta t/2}(t_n + s) &= (1 - s)f^{\Delta t/2}(t_n)  + sQ^+(f^{\Delta t/2}(t_n), f^{\Delta t/2}(t_n)) \,.
\end{split}
\end{equation*}
As can be inferred from the proof of Theorem \ref{t:nanbu}, under Assumption \ref{asm:coll} and for $\b\in\mathcal{P}_\infty(\Theta)$, 
the gain part of the collisional operator is Lipschitz. In particular, it holds for any $f_1,f_2,g_1,g_2\in\mathcal{P}_1(\Rd)$
\begin{equation} \label{eq:Qlip}
\W_1(Q^+(f_1,f_2), Q^+(g_1,g_2)) \leq L_\coll(1 + M_\infty(\b)) \left( \W_1(f_1,g_1) + \W_1(f_2,g_2) \right) \,.
\end{equation}
Therefore, for some $C_1>0$ we have
\begin{equation} \label{eq:A1}
\W_1(\fdt(t_n + s), f^{\Delta t/2}(t_n + s) ) \leq (1 + C_1s) \W_1(\fdt(t_n), f^{\Delta t/2}(t_n ) ) \,.
\end{equation}
Next, we consider $t \in [t_n + \Delta t/2, t_n +\Delta t]$. For $s\in[\Delta t/2, \Delta t]$ we have
\begin{equation*}
\begin{split}
\fdt(t_n + \Delta t/2 + s) &= (1 - s)\fdt(t_n + \Delta t/2)  + sQ^+(\fdt(t_n), \fdt(t_n)) \\
f^{\Delta t/2}(t_n + \Delta t/2 + s) &= (1 - s)f^{\Delta t/2}(t_n + \Delta t/2)  \\
& \qquad + sQ^+(f^{\Delta t/2}(t_n + \Delta t/2), f^{\Delta t/2}(t_n+ \Delta t/2)) 
\end{split}
\end{equation*}
from which follows 
\begin{align*}
\W_1(\fdt(&t_n + \Delta t/2 + s), f^{\Delta t/2}(t_n + \Delta t/2 + s)) \\
& \leq (1 -s) \W_1(\fdt(t_n + \Delta t/2), f^{\Delta t/2}(t_n + \Delta t/2))  \\
&\quad +  s \W_1 \left ( Q^+(\fdt(t_n),\fdt(t_n)), Q^+(f^{\Delta t/2}(t_n + \Delta t/2), f^{\Delta t/2}(t_n+ \Delta t/2))  \right) \\
&  =: (1 - s) A_1 + s A_2\,.
\end{align*}
The term $A_1$ is bounded by \eqref{eq:A1}, while for the term $A_2$, as before, there exists a positive constant $C_2$ such that 
\begin{small}
\begin{align*}
A_2 & \leq C_2 \W_1(\fdt(t_n), f^{\Delta t/2}(t_n+ \Delta t/2))  \\ 
& \leq C_1(1 - \Delta t/2 ) \W_1\left(\fdt(t_n), f^{\Delta t/2}(t_n)\right ) + \Delta t/2 \W_1\big(\fdt(t_n), Q^+(f^{\Delta t/2}(t_n ), f^{\Delta t/2}(t_n))\big ).
\end{align*}\end{small}
By using the linear growth assumption on the collisional map (Assumption \ref{asm:coll}) and Lemma \ref{l:moment}, we have for some $C_3,C_4>0$
\begin{align*}
\W_1(\fdt(t_n), Q^+(f^{\Delta t/2}(t_n ), f^{\Delta t/2}(t_n)) ) &\leq M_1(\fdt(t_n)) \!+\! (1 \!+ \!M_\infty(\b))2 M_1(f^{\Delta t/2}(t_n ))\\
&  \leq C_3 e^{CT} M_1(f_0) \leq C_4\,.
\end{align*}
By collecting the estimates for $A_1,A_2$, we obtain 
\begin{align*}
\W_1(\fdt(t_n + \Delta t/2 + s),& f^{\Delta t/2}(t_n + \Delta t/2 + s))\\
 &\leq (1 - s)(1 + C_1\Delta t/2)\W_1(\fdt(t_n), f^{\Delta t/2}(t_n)) \\
& \qquad  + s C_1 (1 - \Delta t/2 ) \W_1(\fdt(t_n), f^{\Delta t /2}(t_n))   +s \Delta t/2 C_4 \\
& \leq (1 +C_1\Delta t/2)^2 \W_1(\fdt(t_n), f^{\Delta t /2}(t_n)) +  (\Delta t/2)^2 C_4
\end{align*}
and so
$ \sup_{t\in [t_n, t_{n+1}]} \W_1(\fdt(t), f^{\Delta t/2}(t)) \leq e^{C_1 \Delta t} + (\Delta t/2)^2 C_4 \,.
$
After recalling that $\Delta t = \Delta t_k = T/2^{k}$ and $f^{\Delta t_{k}}(0) = f^{\Delta t_{k+1}}(0)$, we iterate the above argument for all time steps to get
\[
\sup_{t\in[0,T]} \W_1(f^{\Delta t_{k}} (t), f^{\Delta t_{k+1}}(t)) \leq C_4 T^2  2^{-k}\,.
\]
Therefore, $\{f^{\Delta t_k}([0,T])\}_{k\in \mathbb{N}}$ is a Cauchy sequence in $\mathcal{C}([0,T],\mathcal{P}_1(\Rd))$.
\end{proof}

Next, we show that Euler scheme converges to the unique solution to \eqref{eq:boltzmann}.

\begin{lemma}\label{l:existance} Under the same settings of Lemma \ref{l:cauchy}, there exists a limit $f\in \mathcal{C}([0,T],\mathcal{P}_1(\Rd))$ to the Euler scheme, which is the unique weak measure solution to \eqref{eq:boltzmann} with initial data $f_0\in \mathcal{P}_1(\Rd)$ in the sense of Definition \ref{def:weaksol}.
\end{lemma}
\begin{proof}
Since $\mathcal{P}_1(\Rd)$ equipped with Wasserstein-1 distance is a complete space, there exists a subsequence such that $f^{\Delta t_k} \to f\in \mathcal{C}([0,T], \mathcal{P}_1(\Rd))$. We also have $f(0) =f_0$. 
In the following, we will use the compact notation $\langle \phi,\mu\rangle := \int\phi(v) \mu(\d v)$ for a test function $\phi$ and measure $f$. Take any $\phi\in C_c^\infty$, from the Kantorovich--Rubinstein duality formula \eqref{eq:wass_dual}, we have for any $t\in [0,T]$
\[
\left |\langle \phi,  \fdt(t) \rangle - \langle \phi , f(t)\rangle \right | \leq  \|\nabla\phi\|_{L^\infty} \W_1(\fdt(t), f(t))\,.
\]
Next, thanks to the Lipschitz property \eqref{eq:Qlip} of $Q^+$ we have
\begin{small}
\begin{align*}
\big| 
& \langle \phi, \fdt(t_{n+1}) - \fdt(t_{n})\rangle  - \int_{t_n}^{t_{n+1}}\langle \phi, Q^+(f(s),f(s)) - f(s) \rangle \d s 
\big|   = \\
&  \qquad = \left| 
\int_{t_n}^{t_{n+1}}\left(  \langle \phi,  Q^+(\fdt(t_n),\fdt(t_n) ) - Q^+(f(s),f(s))  \rangle + \langle \phi, \fdt(t_n) - f(s)\rangle \right) \d s
\right| \\
& \qquad \leq  \|\nabla\phi\|_{L^\infty}\!\!\! \int_{t_n}^{t_{n+1}}\!\!\!\!\!\left(  \W_1\left ( Q^+(\fdt(t_n),\fdt(t_n)), Q^+(f(s),f(s))\right ) + \W_1(\fdt(t_n), f(s)) \right)\! \d s \\
& \qquad  \leq C_0 \int_{t_n}^{t_{n+1}}\W_1(\fdt(t_n), f(s))\d s\,.
\end{align*}\end{small}
From \eqref{eq:eulert}, we note that $\fdt$ is Lipschitz, which leads to 
\begin{small}
\begin{align*}
\int_{t_n}^{t_{n+1}}&\W_1(\fdt(t_n), f(s))\d s   \\
&\leq \int_{0}^{\Delta t } \left(  \W_1(\fdt(t_n), \fdt(t_n + s)) +  \W_1(\fdt(t_n + s), f(t_n + s))  \right) \d s \\
& \leq \int_{0}^{\Delta t} \Big ( s C_1 + \sup_{t\in[0,T]} \W_1(\fdt(t),f(t)) \Big ) \d s  \leq C_2 \Big( \Delta t^2 + \Delta t \sup_{t\in[0,T]} \W_1(\fdt(t),f(t)) \Big) \,.
\end{align*}\end{small}
Finally, consider as before $\Delta t_k = T/2^k$. By using the above estimates, we have
\begin{small}
\begin{align*}
\Big |&\langle \phi,  f(T) -f(0)\rangle -  \int_{0}^T\langle\phi, Q^+(f(s), f(s)) - f(s)  \rangle  \d s
\Big |  \\
&\quad \leq \Big |\langle \phi,  f(T) -f(0)\rangle -  \int_{0}^T\langle\phi, Q^+(f(s), f(s)) - f(s)  \rangle  \d s \\
& \quad\qquad - \langle \phi ,  f^{\Delta t_k}(T) - f^{\Delta t_k}(0)\rangle - \sum_{n = 0}^{2^k - 1}\langle \phi, f^{\Delta t_k}(t_{n+1}) - f^{\Delta t_k}(t_n)) 
\Big | \\
& \quad\leq C_2 \sup_{t\in[0,T]} \W_1(f^{\Delta t_k}(t),f(t)) + \sum_{n=0}^{2^k-1}\Big( \Delta t_k^2 + \Delta t_k \sup_{t\in[0,T]} \W_1(f^{\Delta t_k}(t),f(t))\Big )  \\
&\quad \leq C_2 (T + 1) \Big( \sup_{t\in[0,T]} \W_1(f^{\Delta t_k}(t),f(t)) + \Delta t\Big)\,.
\end{align*}
\end{small}
Since as $k \to \infty$, we have  $\sup_{t\in[0,T]} \W_1(f^{\Delta t_k}(t),f(t))\to 0$, $\Delta t \to 0$, and the above upper bound converges to 0. We can conclude that   $f\in \mathcal{C}([0,T],\mathcal{P}_1(\Rd))$ is a weak measure solution to \eqref{eq:boltzmann} in the sense of Definition \ref{def:weaksol}.

Similar computations also lead to uniqueness of the solution. Let $f_1,f_2$ be two weak measure solutions with initial data $f_0$. At time $t\in[0,T]$, we have for some $\phi\in\textup{Lip}_1(\Rd)$
\begin{align*}
\W_1(f_1(t), f_2(t)) & = \langle\phi, f_1(t) - f_2(t)\rangle \\
& + \int_0^t\!\!\! \left( \langle\phi, Q^+(f_1(s),f_1(s)) - Q^+(f_2(s),f_2(s)) \rangle + \langle\phi,f_1(s)- f_2(s) \rangle\right) \d s \\
& \leq C \int_0^t \W_1(f_1(s), f_2(s))\d s\,.
\end{align*}
By Gr\"onwall's inequality we can conclude that $\W_1(f_1(t),f_2(t)) = 0$ if $f_1(0) = f_2(0)$.
\end{proof}

Finally, we provide a proof to Theorem \ref{t:euler}.

\begin{proof}[Proof of Theorem \ref{t:euler}]

From Lemma \ref{l:existance}, we have existence and uniqueness of a weak measure solution $f$ to \eqref{eq:boltzmann} with initial data $f_0\in \mathcal{P}_1(\Rd)$.
We are left to show that $f\in \textup{Lip}([0,T],\mathcal{P}_1(\Rd))$ and that $\sup_{t\in[0,T]}\W_1(\fdt(t), f(t)) \leq C_{\textup{FE}}\Delta t$. 

First of all, since convergence in Wasserstein-1 distance implies convergence of the first moments^^>\cite{villani2009}, from Lemma \ref{l:moment} we have 
\[
M_1(f(t)) \leq \exp(CT)M_1(f_0) \quad \textup{for all}\;\; t\in [0,T]\,.
\]
Thanks to the linear growth assumption on $\coll(v,v_*,\theta)$ we also have for $\phi \in \mathcal{C}_c^\infty(\Rd)$
\[
\phi(v') - \phi(v) \leq  C \|\nabla\phi\|_{L^\infty} \left( |v| + |v_*|\right)\,.
\]
Next, since $f$ is a solution, for $0\leq s\leq t\leq T$ it holds
\begin{align*}
\langle \phi, f(t) - f(s)\rangle &  = \int_s^t \langle \phi, Q^+(f(\tau),f(\tau)) - f(\tau) \rangle \d \tau \\
& = \frac12\int_s^t\!\!  \iiint \!\!\!\b(\d \theta) \left(\phi(v') + \phi(v'_*) - \phi(v) - \phi(v_*) \right) f(\tau,\d v) f(\tau, \d v_*) \d \tau\\
& \leq 2 C \|\nabla\phi\|_{L^\infty}  \int_s^t M_1(f(\tau))\d \tau \leq |t - s| 2C\|\nabla\phi\|_{L^\infty}  \exp(CT)M_1(f_0)\,.
\end{align*}
By taking the limit of $\phi$ being the test function which realizes the Wasserstein-1 distance, it follows that $\W_1(f(t), f(s)) \leq |t - s| C$ for some $C>0$. Similar computations as above lead to 
\begin{small}
\begin{align*}
\W_1(\fdt(t_{n+1}), f(t_{n+1})) &\leq \W_1(\fdt(t_n), f(t_n)) + C_1\int_0^{\Delta t } \W_1(\fdt(t_n), f(t_n+s))\d s\\
& \leq \W_1(\fdt(t_n), f(t_n)) + \Delta t C_1 \W_1(\fdt(t_n), f(t_n)) + C \int_0^{\Delta t} s  \d s \\
&  = (1 + C_1\Delta t) \W_1(\fdt(t_n), f(t_n)) + C_2\Delta t^2/2
\end{align*}
\end{small}
for $t_n = n\Delta t \in [0,T]$, where we used the Lipschitz continuity of $Q^+$ \eqref{eq:Qlip} and of $f$. By iterating the estimate and by using the fact that $\fdt(0) = f(0) =f_0$, we have
$
\W_1(\fdt(t_n), f(t_n)) \leq C_3 \Delta t.
$
Since $f$ is Lipschitz, we can conclude that for an arbitrary $t\in[0,T]$ it holds
\begin{align*}
\W_1(\fdt(t),f(t)) &\leq \W_1(\fdt(\Delta t \lfloor t/\Delta t\rfloor), f(\Delta t \lfloor t/\Delta t\rfloor)) + C \Delta t  \leq (C_3 + C) \Delta t \,.
\end{align*}
\end{proof}

\section{Outlook}
\label{sec:end}
In this work we developed a novel mathematical framework which offers a deeper theoretical understanding of Monte Carlo methods for Boltzmann equations. This permits to obtain sharp convergence rates in the Wasserstein-1 metric for a general class of models which includes both classical examples from physics and engineering to novel applications in social sciences, life sciences and data science.  

This opens several promising directions for future research, both in terms of rigorous analysis and of developments of novel methods.
A natural next step is to apply this framework to the homogeneous Boltzmann equation, particularly for widely used collision kernels like the variable hard sphere (VHS) model and those relevant to DSMC methods^^>\cite{wagner2005book,wagner2000counter}. The main difficulty lies in the non-Lipschitz nature of the collision map (see Remark \ref{rmk:boltzmann}), which may be addressed using Lipschitz-type estimates and coupling techniques such as Tanaka’s trick^^>\cite{fournier2016,cortez2018}. Current research is also extending the analysis to exactly conservative Monte Carlo methods, like the Nanbu--Babovsky algorithm. Due to the introduction of intrinsic correlations among particles, in fact, the present analysis does not apply directly.

Further extensions include space-velocity kinetic models, such as the full Boltzmann equation, which are essential for realistic system modeling. Non-homogeneous traffic flow models also fall into this category, where dynamics often depend asymmetrically on vehicle positions^^>\cite{tosin2019traffic,puppo2017traffic}.

Lastly, ongoing research focuses on signed particle methods, which incorporate negative weights to reduce variance and computational cost in Monte Carlo simulations^^>\cite{caflish2015negative,bertaglia2024gradient}. Recent advances in Wasserstein distances for signed measures^^>\cite{piccoli2019signed} provide a solid analytical basis for studying and designing new, efficient particle methods within this extended framework.

\section*{Acknowledgments}
The research has been supported by the Royal Society under the Wolfson Fellowship “Uncertainty quantification, data-driven simulations and learning of multiscale complex systems governed by PDEs”. This work has been written within the activities of GNCS group of INdAM (Italian National Institute of High Mathematics). L.P. also acknowledges the  partial support 
by European Union -
NextGenerationEU through the Italian Ministry of University and Research as
part of the PNRR – Mission 4 Component 2, Investment 1.3 (MUR Directorial
Decree no. 341 of 03/15/2022), FAIR “Future” Partnership Artificial Intelligence Research”, Proposal Code PE00000013 - CUP DJ33C22002830006) and by MIUR-PRIN Project 2022, No. 2022KKJP4X “Advanced numerical methods for time dependent parametric partial differential equations with applications”.

\bibliographystyle{abbrv}
\bibliography{bibfile}
\end{document}